\documentclass[12pt,leqno,oneside]{amsart}
\usepackage{mathrsfs,dsfont}
\usepackage{amsmath,amstext,amsthm,amssymb}
\usepackage{charter}
\usepackage{typearea}
\usepackage{commath}
\usepackage{pdfsync}
\usepackage{hyperref}
\usepackage{color}

\usepackage[width=6.4in,height=8.5in]
{geometry}

\newtheorem{Theorem}{Theorem}[section]

\newtheorem{Lemma}[Theorem]{Lemma}
\newtheorem{Corollary}[Theorem]{Corollary}
\theoremstyle{definition}

\newcommand{\db}{\overline\partial}

\newcommand{\wi}{\widetilde}
\DeclareMathOperator{\ric}{Ric}

\DeclareMathOperator{\supp}{supp}

\DeclareMathOperator{\dist}{dist}
\DeclareMathOperator{\re}{Re}

\newcommand{\cali}[1]{\mathscr{#1}}
\newcommand{\cO}{\cali{O}} \newcommand{\cE}{\cali{E}}

\newcommand{\cC}{\cali{C}}
\newcommand{\field}[1]{\mathbb{#1}}

\newcommand{\R}{\field{R}}
\newcommand{\C}{\field{C}}
\newcommand{\N}{\field{N}}

\newcommand{\boldsym}[1]{\boldsymbol{#1}}
\newcommand\bb{\boldsym{b}}
\newcommand{\comment}[1]{}

\begin{document}

\title[On the first order asymptotics of partial Bergman kernels]
{On the first order asymptotics of partial Bergman kernels}
\author{Dan Coman}
\thanks{D.\ Coman is partially supported by the NSF Grant DMS-1300157}
\address{Department of Mathematics, 
Syracuse University, Syracuse, NY 13244-1150, USA}
\email{dcoman@syr.edu}
\author{George Marinescu}
\address{Univerisit\"at zu K\"oln, Mathematisches institut,
Weyertal 86-90, 50931 K\"oln, Germany 
\newline\mbox{\quad}\,Institute of Mathematics `Simion Stoilow', 
Romanian Academy, Bucharest, Romania}
\email{gmarines@math.uni-koeln.de}
\thanks{G.\ Marinescu acknowledges the support of the Syracuse University,
where part of this paper was written.}
\thanks{Funded through the Institutional Strategy of the University of Cologne within the German Excellence Initiative}
\date{December 26, 2015}

\pagestyle{myheadings}

\begin{abstract} 
We show that under very general assumptions the partial Bergman kernel function of sections
vanishing along an analytic hypersurface has exponential decay in 
a neighborhood of the vanishing locus.
Considering an ample line bundle, we obtain a uniform estimate of the Bergman kernel function associated to a singular metric along the hypersurface.
Finally, we study the asymptotics of the partial Bergman kernel function
on a given compact set and near the vanishing locus.
\end{abstract}

\maketitle
\tableofcontents

\section{Introduction}\label{S:intro}
Partial Bergman kernels were recently studied in different contexts, especially
K\"ahler geometry \cite{RS:13,PoSi:14, RWN:14} or random polynomials
\cite{Be07,SZ04}.

Let us consider the following general setting.

\medskip

(A) $(X,\omega)$ is a compact Hermitian manifold of dimension $n$, $\Sigma$ is a smooth analytic hypersurface of $X$, and $t>0$ is a fixed real number.

\smallskip

(B) $(L,h)$ is a singular Hermitian holomorphic line bundle on $X$ with singular metric $h$ which has locally bounded weights.

\medskip

\noindent
We define the space
\begin{equation}\label{e:s1.1}
H^0_0(X,L^p):=H^0\big(X, L^p\otimes\mathcal{O}\big(-\lfloor tp\rfloor\Sigma\big)\big)
\end{equation}
of holomorphic sections of the $p$-th tensor power $L^p$ vanishing to order at least $\lfloor tp\rfloor$
along $\Sigma$, where $\lfloor x\rfloor$ denotes the integral part of $x\in\R$. 
Set $d_{p}=\dim H^0(X,L^p)$ and $d_{0,p}=\dim H^0_0(X,L^p)$.
We introduce on $H^0(X,L^p)$ the $L^2$ inner product $(\cdot\,,\cdot)_p$
induced by the 
metric $h_p=h^{\otimes p}$ and the volume form $\omega^n/n!$\,, see \eqref{e:ip}. 
This inner product is inherited by $H^0_0(X,L^p)$. 
The (full) Bergman kernel function is defined by taking an orthonormal basis 
$\{S_j^p:1\leq j\leq d_{p}\}$ of $(H^0(X,L^p),(\cdot\,,\cdot)_p)$ and setting
\begin{equation}\label{e:Bk}
P_{p}(x)=\sum_{j=1}^{d_{p}}|S^p_j(x)|_{h_p}^2\,,\;\;|S^p_j(x)|_{h_p}^2:=
\langle S_j^p(x),S_j^p(x)\rangle_{h_p},\;x\in X.
\end{equation}
By considering an orthonormal basis $\{S_j^p:1\leq j\leq d_{0,p}\}$
of $(H^0_0(X,L^p),(\cdot\,,\cdot)_p)$, we define the \emph{partial Bergman kernel function} $P_{0,p}$ by 
\begin{equation}\label{e:pBk}
P_{0,p}(x)=\sum_{j=1}^{d_{0,p}}|S^p_j(x)|_{h_p}^2\,,\;x\in X.
\end{equation}
Note that this definition is independent of the choice of basis, cf.\ \eqref{e:pPvar}. 

The asymptotics of the Bergman kernel function for a positive line bundle $(L,h)$ 
\cite{Catlin,Zelditch98}, see also \cite{MM07} for a comprehensive study, 
is very important in understanding the Yau-Tian-Donaldson conjecture. 
On the other hand, partial Bergman kernels are useful in connection to 
the slope semi-stability with respect to a submanifold \cite{RoTh:06}. 
On a toric variety $X$ (and for a toric $\Sigma$) 
this study was carried out in \cite{PoSi:14}. 
In this context it is shown that the partial Bergman kernel has an asymptotic expansion, 
having rapid decay of order $p^{-\infty}$ in a neighborhood $U(\Sigma)$ of $\Sigma$, 
and giving the full Bergman kernel function to order $p^{-\infty}$ outside the closure of 
$U(\Sigma)$. Moreover \cite{PoSi:14} gives a complete distributional asymptotic expansion 
on $X$, whose
leading term has an additional Dirac delta measure plus a dipole measure
over $\partial U(\Sigma)$. These results were generalized in \cite{RS:13} 
and \cite{ZZ16} to the case
when the data in question are invariant under an $S^1$-action. 

In general, if no symmetry is assumed, it was shown in \cite[Theorem\,4.3]{Be07}
that if the bundle $L\otimes\cO(-\Sigma)$ is ample, there exists a neighborhood
$U(\Sigma)$ of $\Sigma$, such that $P_{0,p}(x)$ has exponential decay on $U(\Sigma)$
and $p^{-n}P_{0,p}(x)$ converges to $c_1(L,h)^n/\omega^n$ in $L^1$ outside the closure of 
$U(\Sigma)$. 

Our first result is that under the very general hypotheses (A) and (B) above 
(in particular, without any positivity condition), 
the partial Bergman kernel function decays exponentially in a neighborhood of the divisor $\Sigma$.
\begin{Theorem}\label{T:mt1}
Assume that conditions (A)-(B) are fulfilled.
Then there exist a neighborhood $U_t$ of $\Sigma$ and a constant $a\in(0,1)$ such
that $P_{0,p}\leq a^p$ on $U_t$ for $p>2t^{-1}$. 
In particular $P_{0,p}=O(p^{-\infty})$ as $p\to\infty$ on $U_t$.
\end{Theorem}
For more precise statements see Theorem \ref{T:expd} and Corollary \ref{C:expd}. 
Theorem \ref{T:mt1} can be formulated for non-compact manifolds $X$ 
(see Theorem \ref{T:expdnc}), in which case the exponential decay 
of the partial Bergman kernel holds in a neighborhood of the intersection of $\Sigma$
with any given compact subset of 
$X$. 
This includes for instance the case of classical Bergman spaces 
of $L^2$-holomorphic functions on domains in ${\mathbb C}^n$. 

\medskip

An object which is closely related to the partial Bergman kernel is the
Bergman kernel for a singular metric. The full asymptotic expansion
on compact subsets of the regular part of the metric was established in 
\cite[Theorem\,1.8]{HsM:14}. We are here concerned with asymptotics at arbitrary points with dependence
on the distance to the singular set. More precisely, we will consider the following
situation. 

Let $S_\Sigma\in H^0(X,\mathcal{O}(\Sigma))$ be a canonical holomorphic section of
the line bundle $\mathcal{O}(\Sigma)$, vanishing to first order on $\Sigma$.
We fix a smooth Hermitian metric $h_\Sigma$ on $\mathcal{O}(\Sigma)$ such that
\begin{equation}\label{e:rho}
\varrho:=\log\big|S_\Sigma\big|_{h_\Sigma}<0\:\:\text{on $X$}.
\end{equation}
We consider a function $\xi:X\to\R\cup\{-\infty\}$, smooth on $X\setminus\Sigma$, 
such that $\xi=t\rho$ in a neighborhood $U$ of $\Sigma$. 
Let $\dist(\cdot\,,\cdot)$ be the distance on $X$ induced  by $\omega$. Our main
result is the following:
\begin{Theorem}\label{T:mt2}
Let $(X,\omega),(L,h),\Sigma$ be as in (A)-(B), and assume $\omega$ is K\"ahler and $h$ is smooth. Consider the singular Hermitian metric $\widetilde{h}=he^{-2\xi}$ on $L$ and assume that $c_1(L,\widetilde{h})\geq\varepsilon\omega$ for some constant $\varepsilon>0$. Let $\widetilde{P}_p$ be the Bergman kernel function of $H^0_{(2)}(X,L^p,\widetilde{h}_p,\omega^n/n!)$, where $\widetilde{h}_p:=\widetilde{h}^{\otimes p}$. There exists a constant $C>1$ such that for every $x\in X\setminus\Sigma$ and every $p\in\N$ with $p\dist(x,\Sigma)^{8/3}>C$ we have
\begin{equation}\label{e:Bexp00}
\left|\frac{\widetilde{P}_p(x)}{p^n}\,\frac{\omega^n_x}{c_1(L,\widetilde{h})^n_x}-1\right|
\leq Cp^{-1/8}\,.
\end{equation}
\end{Theorem}
Theorem \ref{T:mt2} can be interpreted in two ways. First, if $x$ runs in a compact set 
$K\subset X\setminus\Sigma$, we have a concrete bound $p_0=C\dist(K,\Sigma)^{-8/3}$
such that for $p>p_0$ the estimate \eqref{e:Bexp00} holds.
By \cite[Theorem\,1.8]{HsM:14} we have 
$\widetilde{P}_p(x)=\sum_{r=0}^\infty\bb_r(x)p^{n-r}+O(p^{-\infty})$
as $p\to\infty$ locally uniformly on $X\setminus\Sigma$. Hence, there exists $p_0(K)\in\N$ and $C_K$
such that for $p>p_0(K)$ we have 
\[\left|\frac{\widetilde{P}_p(x)}{p^n}\,\frac{\omega^n_x}{c_1(L,\widetilde{h})^n_x}-1\right|
\leq C_Kp^{-1}\:\:\text{on $K$.}\] 
However, $p_0(K)$ is not easy to determine.

We can also recast Theorem \ref{T:mt2} as a uniform estimate in $p$ for the singular Bergman
kernel on compact sets of $X\setminus\Sigma$ whose distance to $\Sigma$ decreases as $p^{-3/8}$. 
Indeed, 
set $K_p=\{x\in X: \dist(x,\Sigma)\geq(C/p)^{3/8}\}$.  
Then \eqref{e:Bexp00} holds on $K_p$ for every $p$. 

We consider now the global behavior of the partial Bergman kernel. 
Given a compact set $K\subset X\setminus\Sigma$ we set
\begin{equation}\label{e:t0}
\begin{split}
t_0(K):=\sup\Big\{t>0:\exists\,\eta\in\cC^\infty(X,[0,1]),\:\supp\eta\subset X\setminus K,\:\text{$\eta=1$ near $\Sigma$},\\
\text{and $c_1(L,h)+t\,dd^c(\eta\varrho)$ is a K\"ahler current on $X$}
\Big\}.
\end{split}
\end{equation}
A consequence of Theorems \ref{T:mt1} and \ref{T:mt2} is the following result about 
the 
asymptotics of the partial Bergman kernel:
\begin{Theorem}\label{T:mt3}
Let $(X,\omega),(L,h),\Sigma$ be as in (A)-(B), and assume $\omega$ is K\"ahler, 
$h$ is smooth, and $c_1(L,h)\geq\varepsilon\omega$ for some constant $\varepsilon>0$. 
Let $K\subset X\setminus\Sigma$ be a compact set and let $t\in(0,t_0(K))$. 
Then there exist constants $C>1$, $M>1$ 
and a neighborhood $U_t$ of $\Sigma$, all depending on $t$, such that for
$x\in U_t$ we have
\begin{gather}\label{e:exp1}
Me^{t\varrho(x)}<1\: \text{and} \: P_{0,p}(x)\leq (Me^{t\varrho(x)})^p\: \text{for $p>2/t$},\\
P_{0,p}(x)\geq\frac{p^n}{C}\exp(2tp\varrho(x))\:\:\text{for $p\dist(x,\Sigma)^{8/3}>C$},
\label{e:exp2}
\end{gather}
where the function $\varrho$ is defined in \eqref{e:rho}. Moreover, we have uniformly on $K$,
\begin{equation}\label{e:exp3o}
P_{0,p}(x)=P_p(x)+O(p^{-\infty})\,,\:\:p\to\infty,
\end{equation}
and in particular,
\begin{equation}\label{e:exp3}
P_{0,p}(x)=\bb_0(x)p^n+\bb_1(x)p^{n-1}+O(p^{n-2})\,,\:\:p\to\infty,
\end{equation}
where
\begin{equation}\label{e:coeff}
\bb_0=\frac{c_1(L,h)^n}{\omega^n},\:\: \bb_1= \frac{\bb_0}{8\pi}\,(r^X-2\Delta\log\bb_0),
\end{equation}
and $r^X$, $\Delta\,$, are the scalar curvature, respectively the Laplacian, of the Riemannian metric
associated to $c_1(L,h)$. 
\end{Theorem}
Hence, \eqref{e:exp1} and \eqref{e:exp2} show that on $U_t$ the exponential decay
estimate for the partial Bergman kernel function is sharp. Moreover, on $K$ the 
partial Bergman kernel function has the same asymptotics as the full Bergman kernel function
up to order $O(p^{-\infty})$. This was established in \cite[Theorem 1.1]{RS:13} 
under the additional assumption that there is an $S^1$-action in a neighborhood of $\Sigma$. 
Our method is to estimate the partial Bergman kernel $P_{0,p}$ by above and below with the 
full Bergman kernel $P_{p}$ and the singular Bergman kernel $\wi{P}_{p}$.
On the set where the singular metric $\wi{h}$ equals $h$, the kernels $\wi{P}_{p}$
and  $P_{p}$ differ by $O(p^{-\infty})$. This is shown in Theorem \ref{T:bkloc}, which gives a 
general localization result for singular Bergman kernels. Theorem \ref{T:bkloc} is a straightforward
consequence of \cite{HsM:14}.

However, in Theorem \ref{T:mt3} we do not necessarily obtain a \emph{partition} 
of the manifold $X$ in two sets, one with exponential decay \eqref{e:exp1} 
and one with ``full asymptotics'' \eqref{e:exp3o}, 
since in general $U_t\cup K\neq X$. In \cite{Be07,RS:13,PoSi:14,ZZ16} 
a partition with two different regimes was exhibited under further hypotheses. 
The approach introduced by Berman \cite[Section 4.1]{Be07} 
was to consider the equilibrium metric (or extremal envelope) $h_t$ of $h$ 
with poles along $\Sigma$ (see also \cite{RS:13}). 
The metric $h_t$ exists for $t$ sufficiently small thanks to the positivity of $(L,h)$. The local plurisubharmonic (psh) potentials of $h_t$ have Lelong number $t$ along $\Sigma$. 
It is shown in \cite[Proposition 2.13]{RS:13} that the partial Bergman kernel function $P_{0,p}$ 
has exponential decay in the forbidden region $\{h_t>h\}$ which is a neighborhood of $\Sigma$. 
Moreover, under additional symmetry assumptions, it is shown in \cite[Theorem 1.1]{RS:13} 
that $P_{0,p}$ is essentially equal to the full Bergman kernel outside the forbidden region.

In Theorem \ref{T:mt1} we show that exponential decay holds near $\Sigma$ with no 
assumption on the positivity of $(L,h)$. 
Without positivity assumptions the equilibrium metric
might not exist, but we give here a proof in the general case
based only on the sub-average inequality for holomorphic functions.
  
Our approach in Theorem \ref{T:mt3} differs 
from the envelope approach above in that we first fix a compact $K$ disjoint from $\Sigma$ and then 
construct an interval of small $t>0$ for which the corresponding partial Bergman kernel 
$P_{0,p}$ decays exponentially near $\Sigma$ and is essentially 
equal to the full Bergman kernel on $K$.

\section{Preliminaries}\label{s:prelim}

\subsection{Bergman kernel function} Let $(L,h)$ be a singular Hermitian holomorphic 
line bundle over  a compact Hermitian 
manifold $(X,\omega)$. We denote by $H^0(X,L^p)$ the space of holomorphic sections 
of  $L^p:=L^{\otimes p}$.

Let $H^0_{(2)}(X,L^p)=H^0_{(2)}(X,L^p,h_p,\omega^n/n!)$ 
be the Bergman space of $L^2$-holomorphic sections of 
$L^p$ relative to the metric $h_{p}:=h^{\otimes p}$ 
induced by $h$ and the volume form $\omega^n/n!$ on $X$, endowed with the inner product
\begin{equation}\label{e:ip}
(S,S')_{p}:=\int_{X}\langle S,S'\rangle_{h_{p}}\,\frac{\omega^n}{n!}\,,\;\,S,S'\in H^0_{(2)}(X,L^p).
\end{equation}
Set $\norm{S}_{p}^2=(S,S)_{p}$, $d_{p}=\dim H^0_{(2)}(X,L^p)$.
If $h$ has locally bounded weights (e.\,g.\ $h$ is smooth) 
we have of course $H^0_{(2)}(X,L^p)=H^0(X,L^p)$.
We have the following variational characterization of the partial Bergman kernel
 \begin{equation}\label{e:pPvar}
P_{0,p}(x)=\max\Big\{|S(x)|^2_{h_p}:\,S\in H^0_{0}(X,L^p),\;\|S\|_p=1\Big\},
\end{equation}
and similar characterizations hold for the full and singular Bergman kernel functions $P_p$
and $\wi{P}_p$.

Throughout the paper we also use the following terminology. 
For a sequence of continuous functions $f_p$ on a manifold $M$ 
we write $f_p=O(p^{-\infty})$ if for every compact subset $K\subset M$ 
and any $\ell\in\N$ there exists $C_{K,\ell}>0$ such that for all $p\in\N$ 
we have $\|f_p\|_K\leq C_{K,\ell}\,p^{-\ell}$.

\subsection{Geometric set-up}\label{ss:geoset} We prepare here the 
geometric set-up needed for the proofs of our results, by constructing a 
special neighborhood $W$ of $\Sigma$. 

\par Let \((X,\omega)\) be a compact Hermitian manifold of dimension $n$. 
Let $(U,z)$, $z=(z_1,\ldots,z_n)$, be local coordinates centered at a point $x\in X$. 
For $r>0$ and 
$y\in U$ we denote by
\[\Delta^n(y,r)=\{z\in U: |z_j-y_j|\leq r,\:j=1,\ldots,n\}\]
the (closed) polydisk of polyradius $(r,\ldots,r)$ centered at $y$. 
If $\omega$ is a K\"ahler form, the coordinates $(U,z)$ are called K\"ahler at $y\in U$ if 
$$\omega_z=\frac i2\sum_{j=1}^{n}dz_j\wedge d\overline{z}_j+O(|z-y|^2)
\:\:\text{on \(U\)}.$$

Since $\Sigma$ is compact, we can find an open cover 
$\mathcal W=\{W_j\}_{1\leq j\leq N}$ of $\Sigma$, where $W_j$ 
are Stein contractible coordinate neighborhoods centered at points 
$y_j\in\Sigma$, such that 
\begin{equation}\label{e:Wgeom}
\begin{split}
&\Delta^n(y_j,2)\subset W_j\,,\,\;\Sigma\subset W:=\bigcup_{j=1}^{N}\Delta^n(y_j,1),\\
&\Sigma\cap W_j=\big\{z\in W_j:z_1=0\big\}\,,
\:\text{for $j=1\ldots,N$},
\end{split}
\end{equation}
where $z=(z_1,\ldots,z_n)$ are the coordinates on $W_j$. 
Moreover, if $\omega$ is a K\"ahler form, we may also ensure that 
\begin{equation}\label{e:WgeomK}
\forall\,x\in\Delta^n(y_j,1),\;\exists\,z=z(x) 
\text{ coordinates on \(\Delta^n(y_j,2)\) centered at \(x\) and K\"ahler at \(x\).}
\end{equation}

As in \cite[\S2.5, Lemma 2.7]{CMM14} one can easily prove the following: 

\begin{Lemma}\label{L:rc}
Let $(X,\omega),(L,h),\Sigma,\wi{h}$ be as in Theorem \ref{T:mt2}, 
and let $\mathcal W=\{W_j\}_{1\leq j\leq N}$ be an open cover of 
$\Sigma$ verifying \eqref{e:Wgeom} and \eqref{e:WgeomK}. 
There exist constants \(C_1>1\), \(C_2>0\) and \(r_1>0\) 
with the following property: if \(j\in\{1,\ldots,N\}\), \(x\in\Delta^n(y_j,1)\) 
and \(z=z(x)\) are the coordinates on \(\Delta^n(y_j,2)\) 
given by \eqref{e:WgeomK}, then:

\par (i) $\Delta_z^n(x,r_1)\Subset\Delta^n(y_j,2)$ and for $r\leq r_1$ we have 
\begin{equation}\label{e:dm}
n!\,dm\leq(1+C_1r^2)\omega^n\,, \:\:
\omega^n\leq(1+C_1r^2)n!\,dm\,\:\:\text{on $\Delta_z^n(x,r)$},
\end{equation}
where \(dm=dm(z)\) is the Euclidean volume and $\Delta_z^n(x,\cdot)$ is the 
open polydisk relative to the coordinates $z$.
  
\par (ii) \((L,\wi{h})\) has a weight \(\varphi_x\) on \(W_j\) with 
\begin{equation}\label{e:phix}
\begin{split}
&\varphi_x=t\log|f|+\psi_x\,,\quad \psi_x\in\cC^\infty(W_j),\\
&\psi_x (z) = \re F_x(z)+ \psi_x^\prime(z) + \widetilde\psi_x(z)\,\text{ on } \Delta^n(y_j,2),
\end{split}
\end{equation}
where $f$ is a defining function for $\Sigma\cap W_j$, $F_x(z)$ is a holomorphic polynomial 
of degree $\leq 2$ in $z$, $\psi_x^\prime(z)= \sum_{\ell=1}^n \lambda_{\ell} |z_\ell|^2$,
$\lambda_{\ell}=\lambda_{\ell}(x)$,
and 
\begin{equation}\label{e:psix}
|\widetilde\psi_x(z)|\leq C_2|z|^3\,,\:\:z\in\Delta_z^n(x,r_1)\,.
\end{equation}
\end{Lemma}

\section{Exponential decay}\label{S:expd}

We prove here Theorem \ref{T:mt1}. Let $\mathcal W=\{W_j\}_{1\leq j\leq N}$ 
be the cover of $\Sigma$ and $W\supset\Sigma$ be the neighborhood of $\Sigma$ 
constructed in section \ref{ss:geoset} (see \eqref{e:Wgeom}). 
For a function $\varphi\in L^\infty_{loc}(W_j)$ 
set 
$$\|\varphi\|_{\infty,W_j}=\sup\big\{|\varphi(w)|:\,w\in\Delta^n(y_j,2)\big\}.$$
Let \((L,h)\) be a singular Hermitian holomorphic line bundle on \(X\), 
where the metric $h$ has locally bounded weights. Since \(L|_{W_j}\) is trivial, 
we fix a holomorphic frame $e_j$ of $L|_{W_j}$, and denote by \(\varphi_j\) 
the corresponding weight of \(h\) on \(W_j\), i.e. \(|e_j|_h=e^{-\varphi_j}\). Set
\begin{equation}\label{e:h-infty}
\|h\|_\infty=\|h\|_{\infty,\mathcal W}:=\max\big\{1,\|\varphi_j\|_{\infty,W_j}:
\,1\leq j\leq N\big\},
\end{equation}
and let $\varrho$ be the function defined in \eqref{e:rho}.

\begin{Theorem}\label{T:expd}
In the setting of Theorem \ref{T:mt1}, there exists a constant $A\geq1$ 
depending only on $\rho$ and $\mathcal W$ such that for any 
$S\in H^0_{0}(X,L^p)$, $x\in W$, and $p\geq1$, we have
$$|S(x)|_{h_p}^2\leq (Ae^{\rho(x)})^{2\lfloor tp\rfloor}e^{4p\|h\|_\infty}\|S\|_p^2\,.$$
Therefore, for every $x\in W$ and $p\geq1$, 
$$P_{0,p}(x)\leq (Ae^{\rho(x)})^{2\lfloor tp\rfloor}e^{4p\|h\|_\infty}\,.$$
\end{Theorem}
For the proof we need the following elementary lemma.

\begin{Lemma}\label{L:il}
If $k\geq0$ and $f\in\mathcal{O}(\Delta(0,2))$, where 
$\Delta(0,2)\subset\mathbb C$ is the closed disk centered at $0$ and of radius $2$, then
$$\int_{\Delta(0,2)}|f(\zeta)|^2\,dm(\zeta)
\leq\frac{k+1}{2^{2k}}\int_{\Delta(0,2)}|\zeta|^{2k}|f(\zeta)|^2\,dm(\zeta)\,.$$
\end{Lemma}
\begin{proof}
Consider the power expansion $f(\zeta)=\sum_{j=0}^\infty a_j\zeta^j$
of $f$ in $\Delta(0,2)$. Integrating in polar coordinates we obtain
\[
\int_{\Delta(0,2)}|f(\zeta)|^2\,dm(\zeta)=
2\pi\sum_{j=0}^\infty |a_j|^2\int_0^2r^{2j+1}\,dr=
2\pi\sum_{j=0}^\infty\frac{2^{2j+2}}{2j+2}\,|a_j|^2\,.
\] 
On the other hand, $\zeta^kf(\zeta)=\sum_{j=k}^\infty a_{j-k}\zeta^j$, so
\[
\begin{split}
\int_{\Delta(0,2)}|\zeta|^{2k}|f(\zeta)|^2\,dm(\zeta)
&=2\pi\sum_{j=k}^\infty\frac{2^{2j+2}}{2j+2}\,|a_{j-k}|^2
=2\pi\sum_{j=0}^\infty\frac{2^{2j+2+2k}}{2j+2+2k}\,|a_{j}|^2\\
&\geq \frac{2^{2k}}{k+1}\,2\pi\sum_{j=0}^\infty\frac{2^{2j+2}}{2j+2}\,|a_{j}|^2
=\frac{2^{2k}}{k+1}\,\int_{\Delta(0,2)}|f(\zeta)|^2\,dm(\zeta)\,.
\end{split}
\] 
\end{proof}

\begin{proof}[Proof of Theorem \ref{T:expd}]
Let $x\in W$. Fix $j\in\{1,\ldots,N\}$ such that $x\in \Delta^n(y_j,1)$ 
and let $e_j$ be the local frame of $L|_{W_j}$ and $\varphi_j$ be the 
corresponding weight of $h$ as considered in \eqref{e:h-infty}. 
Let $S\in H^0_{0}(X,L^p)$. On $W_j$ we write $S=se_j^{\otimes p}$, 
with $s\in\mathcal{O}(W_j)$. Then
we have $s(z)=z_1^{\lfloor tp\rfloor}\widetilde{s}(z)$, 
with $\widetilde{s}\in\mathcal{O}(W_j)$. Using the sub-averaging inequality we get
\begin{equation}\label{e:es1}
\begin{split}
|S(x)|_{h_p}^2&=|x_1|^{2\lfloor tp\rfloor}|\widetilde{s}(x)|^2 e^{-2p\varphi_j(x)}
\leq |x_1|^{2\lfloor tp\rfloor}e^{-2p\varphi_j(x)} \frac{1}{\pi^n}
\int_{\Delta^n(x,1)}|\widetilde{s}(z)|^2\,dm(z)\\
&\leq |x_1|^{2\lfloor tp\rfloor}e^{-2p\varphi_j(x)} 
\int_{\Delta^n(0,2)}|\widetilde{s}(z)|^2\,dm(z)\,.
\end{split}
\end{equation}
Applying Fubini's theorem for the splitting $z=(z_1,z')$ and Lemma \ref{L:il}
for the variable $z_1$, we obtain 
\begin{equation}\label{e:es2}
\begin{split}
\int_{\Delta^n(0,2)}|\widetilde{s}(z)|^2\,dm(z)&=
\int_{\Delta^{n-1}(0,2)}\int_{\Delta(0,2)}|\widetilde{s}(z_1,z')|^2\,dm(z_1)dm(z')\\
&\leq\frac{\lfloor tp\rfloor+1}{2^{2\lfloor tp\rfloor}}\int_{\Delta^n(0,2)}|z_1|^{2\lfloor tp\rfloor}
|\widetilde{s}(z)|^2\,dm(z)\\
&\leq C\,\exp\left(2p\max_{\Delta^n(0,2)}\varphi_j\right)
\int_{\Delta^n(0,2)}|s(z)|^2 e^{-2p\varphi_j(z)}\,\frac{\omega^n}{n!}\,,
\end{split}
\end{equation}
where $C=C(\mathcal W)\geq1$ is chosen such that 
$dm(z)\leq C\omega^n/n!$ on each $\Delta^n(y_j,2)$  
in the local coordinates of $W_j$, for $j=1,\ldots,N$. Combining \eqref{e:es1} and
\eqref{e:es2} we get
\begin{equation}\label{e:s1}
|S(x)|_{h_p}^2\leq C\,|x_1|^{2\lfloor tp\rfloor}
\exp\left(2p\max_{\Delta^n(0,2)}\varphi_j-2p\varphi_j(x)\right)\|S\|_p^2
\end{equation}
Note that there exists a constant $A'=A'(\rho,W)>1$ such that 
\begin{equation}\label{e:xrho}
|x_1|\leq A'e^{\rho(x)}\,,\:\:x\in W\,.
\end{equation}
Set $A=A'C$. The estimates \eqref{e:s1} and \eqref{e:xrho} yield
$$|S(x)|_{h_p}^2\leq (C|x_1|)^{2\lfloor tp\rfloor}e^{4p\|h\|_\infty}\|S\|_p^2
\leq(Ae^{\rho(x)})^{2\lfloor tp\rfloor}e^{4p\|h\|_\infty}\|S\|_p^2\,.$$
Taking into account \eqref{e:pPvar} we immediately obtain the conclusion. 
\end{proof}
\begin{Corollary}\label{C:expd}
In the setting of Theorem \ref{T:expd} we let
\begin{equation}\label{e:ut}
U_t:=\Big\{x\in W:\,(Ae^{\rho(x)})^{t}\,e^{4\|h\|_\infty}<1\Big\}\,.
\end{equation}
Then for any $x\in U_t$ and $p>2t^{-1}$ we have
\begin{equation}\label{e:estb1}
P_{0,p}(x)\leq \big[(Ae^{\rho(x)})^{t}\,e^{4\|h\|_\infty}\big]^p\,.
\end{equation} 
In particular $P_{0,p}=O(p^{-\infty})$ as $p\to\infty$ on $U_t$.
\end{Corollary}
\begin{proof} This follows immediately from Theorem \ref{T:expd}, 
since $Ae^{\rho(x)}<1$ for $x\in U_t$, and $2\lfloor tp\rfloor>2tp-2>tp$ for $p>2/t$.
\end{proof}

We conclude this section by giving a version of Theorem \ref{T:expd} 
in the case when $X$ is not compact. Let $(X,\omega)$ be a Hermitian 
manifold of dimension $n$, $\Sigma$ be a smooth analytic hypersurface 
of $X$, $t>0$ a fixed real number, and $(L,h)$ a singular Hermitian holomorphic 
line bundle on $X$ with singular metric $h$ which has locally bounded weights. 

As in the case of a compact manifold $X$, we introduce
the Bergman space $H^0_{(2)}(X,L^p)=H^0_{(2)}(X,L^p,h_p,\omega^n/n!)$ 
of $L^2$-holomorphic sections of 
$L^p$ relative to the metric $h_{p}$ 
induced by $h$ and the volume form $\omega^n/n!$ on $X$, endowed with the inner product
\eqref{e:ip}.
Let $H^0_{(2),0}(X,L^p)\subset H^0_{(2)}(X,L^p)$ be the Bergman space of $L^2$-holomorphic sections of  
$L^p$ vanishing to order at least $\lfloor tp\rfloor$ along $\Sigma$, 
\begin{equation*}\label{e:s1.12}
H^0_{(2),0}(X,L^p):=
H^0_{(2)}\big(X, L^p\otimes\mathcal{O}\big(-\lfloor tp\rfloor\Sigma\big)\big).
\end{equation*}
The spaces $H^0_{(2)}(X,L^p)$ and $H^0_{(2),0}(X,L^p)$ are not necessarily finite dimensional 
but their Bergman kernel functions $P_{p}$ and $P_{0,p}$ 
can be defined as in \eqref{e:Bk}, \eqref{e:pBk}, by means of at most countable 
orthonormal bases, see \cite[Lemma 3.1]{CM11}.
 
We fix a compact set $E\subset X$ and consider an open cover 
$\mathcal W=\{W_j\}_{1\leq j\leq N}$ and a neighborhood $W$ of the compact set 
$\Sigma\cap E$ constructed as in \eqref{e:Wgeom}. 
Finally define $\|h\|_\infty=\|h\|_{\infty,\mathcal W}$ as in \eqref{e:h-infty}, 
and the function $\varrho$ as in \eqref{e:rho}, such that $\varrho<0$ 
in a neighborhood of $E$. 
The following theorem is proved exactly as Theorem \ref{T:expd}:


\begin{Theorem}\label{T:expdnc}
Let $(X,\omega)$ be a Hermitian 
manifold of dimension $n$, $\Sigma$ be a smooth analytic hypersurface 
of $X$, $t>0$ a fixed real number, and $(L,h)$ a singular Hermitian holomorphic 
line bundle on $X$ with singular metric $h$ which has locally bounded weights. 
Then for any compact set $E\subset X$ there exists a neighborhood $W$ of the compact set 
$\Sigma\cap E$
and a constant $A\geq1$ such that for every $x\in W$ and $p\geq1$, 
$$P_{0,p}(x)\leq (Ae^{\rho(x)})^{2\lfloor tp\rfloor}e^{4p\|h\|_\infty}\,.$$
The constant $A$ depends only on 
$\rho$ and $\mathcal W$ above.
\end{Theorem}

\section{Singular Bergman kernel}
In this section we prove Theorem \ref{T:mt2} by using ideas of Berndtsson, 
who gave in \cite[Section\,2]{B03} a simple proof for the first order asymptotics 
of the Bergman kernel function in the case of powers of an ample line bundle 
(see also \cite[Theorem 1.3]{CMM14}).

We start by recalling the following version of Demailly's estimates for the 
$\db$ operator \cite[Th\'eor\`eme 5.1]{De:82} (see also \cite[Theorem\,2.5]{CMM14}) 
which will be needed in our proofs.

\begin{Theorem}\label{T:db} Let $(X,\omega)$
    be a compact K\"ahler manifold of dimension $n$,
    and let $B>0$ be a constant such that $\ric_\omega\geq-2\pi B\omega$
    on $X$. Let $(L,h)$ be a singular Hermitian holomorphic line 
    bundle on $X$ such that $c_1(L,h)\geq\varepsilon\omega$, and 
    fix $p_0$ such that $p_0\varepsilon\geq 2B$. Then for all $p>p_0$ and all
    $g\in L^2_{0,1}(X,L^p,loc)$ with $\db g=0$ and 
    $\int_X|g|^2_{h_p}\,\omega^n<\infty$ there exists 
    $u\in L^2_{0,0}(X,L^p,loc)$ such that $\db u=g$ and 
    $\int_X|u|^2_{h_p}\,\omega^n\leq\frac{2}{p\varepsilon}\,
    \int_X|g|^2_{h_p}\,\omega^n$. 
\end{Theorem}

\begin{proof}[Proof of Theorem \ref{T:mt2}]
Let $\mathcal W=\{W_j\}_{1\leq j\leq N}$ be an open cover of $\Sigma$ 
verifying \eqref{e:Wgeom} and \eqref{e:WgeomK}. If \(j\in\{1,\ldots,N\}\) 
and \(x\in\Delta^n(y_j,1)\), let \(z=z(x)\) be the coordinates on \(\Delta^n(y_j,2)\) 
given by \eqref{e:WgeomK}, and let \(e_{j,x}\) be a holomorphic frame of \(L\) on \(W_j\) 
such that \(|e_{j,x}|_{\widetilde{h}}=e^{-\varphi_x}\), where $\varphi_x$ 
is given by \eqref{e:phix}.

Assume now that \(x\in\Delta^n(y_j,1)\setminus\Sigma \) and define 
$$r_x:=\sup\Big\{r\in(0,r_1]:\Delta_z^n(x,r)\subset\Delta^n(y_j,2)\setminus\Sigma\Big\}\,.$$
We have 
\begin{equation}\label{e:omc}
\begin{split}
\omega_x&=\frac{i}{2}\sum_{\ell=1}^n dz_\ell\wedge d\bar{z}_\ell\,,\\
c_1(L,\widetilde{h})_x&=dd^c \varphi_x(0)=dd^c \psi_x(0)
=dd^c \psi^\prime_x(0)=\frac{i}{\pi}
\sum_{\ell=1}^n\lambda_\ell\, dz_\ell\wedge d\bar{z}_\ell\,.
\end{split}
\end{equation}
Since $c_1(L,\widetilde{h})_x\geqslant\varepsilon\omega_x$ it follows that 
$\lambda_\ell\geq\varepsilon$, $\ell=1,\ldots,n$. Moreover, there exists 
$H_x\in\mathcal{O}(\Delta_z^n(x,r_x))$ such that 
$\operatorname{Re} H_x=\operatorname{Re} F_x+t\log|f|$. 
We define a new frame for $L$ over $\Delta_z^n(x,r_x)$ by $e_x=e^{H_x}e_{j,x}$.
Hence
$$|e_x|_{\widetilde{h}}=\exp(\operatorname{Re}H_x)\exp(-\varphi_x)
=\exp(-\psi_x^\prime-\widetilde\psi_x)\,.$$

\par We fix now $j\in\{1,\ldots,N\}$ and $x\in\Delta^n(y_j,1)\setminus\Sigma$ 
and we will estimate $\widetilde{P}_p(x)$.  Let $r_p\in(0,r_x/2)$
be an arbitrary number which will be specified later. 
We start by estimating the norm of a section 
$S \in H^0_{(2)}(X,L^p,\widetilde{h}_p,\omega^n/n!)$ at $x$. 
Writing $S = se_x^{\otimes p}$, 
where $s \in \mathcal{O}(\Delta_z^n(x,r_x))$, we obtain by the 
sub-averaging inequality for psh functions on $\Delta_z^n(x,r_p)=\Delta^n(0,r_p)$,
\[
|S(x)|^2_{\widetilde{h}_p} = |s(0)|^2 \leq \frac{\int_{\Delta^n(0,r_p)} 
|s|^2 e^{-2p \psi^\prime_x}\,dm}{\int_{\Delta^n(0,r_p)} 
e^{-2p \psi^\prime_x} \,dm}\,\cdot
\]
We have further by \eqref{e:dm}, \eqref{e:psix},
\begin{eqnarray*}
\int_{\Delta^n(0,r_p)}|s|^2e^{-2p \psi^\prime_x}\,dm&
\leq&(1+C_1r^2_p)\exp\!\big(2p\sup_{\Delta^n(0,r_p)}
\widetilde{\psi}_x\big)\int_{\Delta^n(0,r_p)}|s|^2
e^{-2p(\psi_x^\prime+ \widetilde\psi_x)}\frac{\omega^n}{n!}\\ 
&\leq&(1+C_1r^2_p)\exp\!\big(2C_2p\,r^3_p\big)
\|S\|^2_p\,.
\end{eqnarray*}
Set 
$$E(r):=\int_{|\xi|\leq r}e^{-2|\xi|^2}\,dm(\xi)
=\frac{\pi}{2}\,\left(1-e^{-2r^2}\right).$$ 
Since 
$\lambda_\ell\geq\varepsilon$ we obtain 
$$\frac{E(r_p\sqrt{p\varepsilon}\,)^n}{p^n\lambda_1\ldots\lambda_n} 
\leq \int_{\Delta^n(0,r_p)}e^{-2p\psi^\prime_x}\,dm
\leq\int_{\C^{n}}e^{-2p\psi^\prime_x}\,dm 
= \frac{(\pi/2)^n}{p^n\lambda_1\ldots\lambda_n}\,\cdot$$
Combining these estimates it follows that 
\begin{equation}\label{e:bk1}
|S(x)|_{\widetilde{h}_p}^2\leq\frac{(1+C_1r_p^2)\exp\!\big(2C_2p\,r_p^3
\big)}{E(r_p\sqrt{p\varepsilon})^n}\,p^n\lambda_1
\ldots\lambda_n\,\|S\|_p^2\,.
\end{equation}
The singular Bergman kernel also satisfies a variational formula,
\[\widetilde{P}_p(x)=\max\Big\{|S(x)|^2_{\widetilde{h}_p}:
\,S\in H^0_{(2)}(X,L^p,\widetilde{h}_p,\omega^n/n!),\;\|S\|_p=1\Big\}.
\]
Hence \eqref{e:bk1} implies the following upper estimate for the
singular Bergman kernel,
\begin{equation}\label{e:bk2}
\frac{\widetilde{P}_p(x)}{p^n\lambda_1\ldots\lambda_n}\leq 
\frac{(1+C_1r_p^2)\exp\!\big(2C_2p\,r_p^3
\big)}{E(r_p\sqrt{p\varepsilon})^n}\,,\quad \forall\, r_p\in(0,r_x/2). 
\end{equation}
\par For the lower estimate of $\widetilde{P}_p$, let $0\leq\chi\leq1$ 
be a smooth cut-off function on $\C^{n}$ with support in $\Delta^n(0,2)$ such that 
$\chi\equiv 1$ on $\Delta^n(0,1)$, and set
$\chi_p(z)=\chi(z/r_p)$. Then $F=\chi_p e_x^{\otimes p}$ is a section of 
$L^p$ and $|F(x)|_{\widetilde{h}_p}=|e_x^{\otimes p}(x)|_{\widetilde{h}_p}=1$. 
We have  
\begin{equation}\label{e:bk3}\begin{split}
\|F\|^2_p&\leq\int_{\Delta^n(0,2r_p)}
e^{-2p(\psi_x^\prime+ \widetilde\psi_x)}\,\frac{\omega^n}{n!}\\
&\leq (1+4C_1r^2_p)\exp\!\big(16C_2p\,r_p^3\big)
\int_{\Delta^n(0,2r_p)}e^{-2p\psi^\prime_x}\,dm\\
&\leq\left(\frac{\pi}{2}\right)^n\,\frac{(1+4C_1r_p^2)
\exp\!\big(16C_2p\,r_p^3\big)}{p^n\lambda_1
\ldots\lambda_n}\,\cdot
\end{split}
\end{equation}
Set $\alpha=\overline\partial F$. Since 
$\|\overline\partial\chi_p\|^2=\|\overline\partial\chi\|^2/r_p^2$, 
where $\|\overline\partial\chi\|$ denotes the maximum of 
$|\overline\partial\chi|$, we obtain as above 
\[
\|\alpha\|_p^2=\int_{\Delta^n(0,2r_p)}|\overline\partial\chi_p|^2
e^{-2p(\psi_x^\prime+ \widetilde\psi_x)}\,\frac{\omega^n}{n!}
\leq\frac{\|\overline\partial\chi\|^2}{r_p^2}\,\left(\frac{\pi}{2}\right)^n\,
\frac{(1+4C_1r_p^2)\exp\!\big(16C_2p\,r_p^3\big)}
{p^n\lambda_1\ldots\lambda_n}\,\cdot
\]
There exists $p_0\in\N$ such that for 
$p>p_0$ we can solve the $\overline\partial$--equation by 
Theorem \ref{T:db}. We get a smooth section $G$ of $L^p$ 
with $\overline\partial G=\alpha=\overline\partial F$ and 
\begin{equation}\label{e:bk4}
\|G\|_p^2\leq \frac{2}{p\varepsilon}\|\alpha\|_p^2
\leq \frac{2\|\overline\partial\chi\|^2}{p\varepsilon r_p^2}\,
\left(\frac{\pi}{2}\right)^n\,\frac{(1+4C_1r_p^2)
\exp\!\big(16C_2p\,r_p^3\big)}{p^n\lambda_1\ldots
\lambda_n}\,\cdot 
\end{equation}
Note that $G$ is holomorphic on $\Delta^n(0,r_p)$ since 
$\overline\partial G=\overline\partial F=0$ there. 
So the estimate \eqref{e:bk1} applies to $G$ on $\Delta^n(0,r_p)$ and gives  
\begin{eqnarray*}
|G(x)|_{\wi{h}_p}^2&\leq&\frac{(1+C_1r_p^2)
\exp\!\big(2C_2p\,r_p^3\big)}
{E(r_p\sqrt{p\varepsilon})^n}\,p^n\lambda_1\ldots\lambda_n\|G\|_p^2\\
&\leq&\frac{2\|\overline\partial\chi\|^2}{p\varepsilon  r_p^2
E(r_p\sqrt{p\varepsilon})^n}\,\left(\frac{\pi}{2}\right)^n\,(1+4C_1r_p^2)^2 
\exp\!\big(18C_2p\,r_p^3\big).
\end{eqnarray*}
Let $S=F-G\in H^0_{(2)}(X,L^p,\wi{h}_p,\omega^n/n!)$. Then
\begin{eqnarray*}
|S(x)|_{\wi{h}_p}^2&\geq&(|F(x)|_{\wi{h}_p}-|G(x)|_{\wi{h}_p})^2
= (1-|G(x)|_{\wi{h}_p})^2\\
&\geq&\left[1-\,\left(\frac{\pi}{2}\right)^{n/2}
\frac{\sqrt{2}\,\|\overline\partial\chi\|(1+4C_1r_p^2)}{r_p\sqrt{p\varepsilon}\,
E(r_p\sqrt{p\varepsilon})^{n/2}}\,\exp\!\big(9C_2p\,r_p^3\big)
\right]^2=:K_1(r_p)\,.
\end{eqnarray*}
Moreover, by \eqref{e:bk3} and \eqref{e:bk4}
$$\|S\|_p^2\leq (\|F\|_p+\|G\|_p)^2
\leq \left(\frac{\pi}{2}\right)^n \,\frac{K_2(r_p)}{p^n\lambda_1
\ldots\lambda_n} \,,$$
where 
$$K_2(r_p)=(1+4C_1r_p^2) \exp\!\big(16C_2p\,r_p^3\big)
\left(1+\frac{\sqrt{2}\,\|\overline\partial\chi\|}{r_p\sqrt{p\varepsilon}}\right)^2
\cdot$$
Therefore 
\begin{equation}
\label{e:bk5}
\wi{P}_p(x)\geq \frac{|S(x)|_{\wi{h}_p}^2}{\|S\|_p^2}
\geq \left(\frac{2}{\pi}\right)^np^n\lambda_1\ldots\lambda_n\,
\frac{K_1(r_p)}{K_2(r_p)}\,\cdot
\end{equation}
Using now \eqref{e:omc}, \eqref{e:bk2} and \eqref{e:bk5} we deduce that for every 
$x\in\bigcup_{j=1}^N\Delta^n(y_j,1)\setminus\Sigma\,$, $r_p<r_x/2$ and $p>p_0$,
\begin{equation}
\label{e:bk6}
\frac{K_1(r_p)}{K_2(r_p)}\leq \wi{P}_p(x)\,\frac{\omega^n_x}{p^nc_1(L,\wi{h})_x^n}
\leq K_3(r_p)\,,\end{equation}
where 
$$K_3(r_p)=\left(\frac{\pi/2}{E(r_p\sqrt{p\varepsilon})}\right)^n
(1+C_1r_p^2)\exp\!\big(2C_2p\,r_p^3\big)\,.$$
We take now $r_p=p^{-3/8}$, so $p\,r_p^3=p^{-1/8}\to0$ and 
$p\,r_p^2=p^{1/4}\to\infty$ as $p\to\infty$. Note that there exists a constant $C_3>0$ such that 
$$K_1(p^{-3/8})\geq 1-C_3p^{-1/8}\,,\,\;K_2(p^{-3/8})
\leq 1+C_3p^{-1/8}\,,\,\;K_3(p^{-3/8})\leq 1+C_3p^{-1/8}\,.$$
It follows by \eqref{e:bk6} that there exists a constant $C_4>0$ such that 
\begin{equation}
\label{e:bk7}
1-C_4\,p^{-1/8}\leq \wi{P}_p(x)\,\frac{\omega^n_x}{p^nc_1(L,\wi{h})_x^n}
\leq 1+C_4\,p^{-1/8}
\end{equation}
holds for every $x\in\bigcup_{j=1}^N\Delta^n(y_j,1)\setminus\Sigma\,$, 
$p^{-3/8}<r_x/2$ and $p>p_0$.
Now $r_x>c\dist(x,\Sigma)$, for some constant $c>0$, so there exists a constant 
$C_5>0$ such that \eqref{e:bk7}
holds for $p>C_5\dist(x,\Sigma)^{-8/3}$. This concludes the proof of 
\eqref{e:Bexp00} for $x\in\bigcup_{j=1}^N\Delta^n(y_j,1)\setminus\Sigma$.

By \cite[Theorem\,1.8]{HsM:14} there exist $C_6>0$ and $p^{\prime}_0\in\N$ such that
$$\left|\wi{P}_p(x)\,\frac{\omega^n_x}{p^nc_1(L,\wi{h})_x^n}-1\right|\leq\frac{C_6}{p}\,,$$
for $x\in X\setminus\bigcup_{j=1}^N\Delta^n(y_j,1)$ and $p>p^{\prime}_0$. 
The proof of Theorem \ref{T:mt2} is complete.
\end{proof}

\section{Estimates for the partial Bergman kernel}
In this section we prove Theorem \ref{T:mt3}. Let $t<t_0(K)$. 
By the definition \eqref{e:t0} of $t_0(K)$,
there exist $\eta\in\cC^\infty(X,[0,1])$ and $\delta>0$ such that 
$\supp\eta\subset X\setminus K$, $\eta=1$ near $\Sigma$ and 
$c_1(L,h)+tdd^c(\eta\varrho)\geq\delta\omega$ in the sense of currents on $X$. Define
\[\wi{h}_t=h\exp(-2t\eta\varrho)\,,\:\:\wi{h}_{t,p}=\wi{h}_t^{\otimes p}\,.\]
Note that $\wi{h}_t=h$ in a neighborhood of $K$ and $\wi{h}_t\geq h$ on $X$.
Since $\Sigma$ is smooth, it follows by \eqref{e:rho} that 
$H^0_0(X,L^p)=H^0_{(2)}(X,L^p,\wi{h}_{t,p},\omega^n/n!)$.
We denote the norm on $H^0_{(2)}(X,L^p,\wi{h}_{t,p},\omega^n/n!)$ by
\[\|S\|_{t,p}^2=\int_X|S|_{\wi{h}_{t,p}}^2\frac{\omega^n}{n!}
=\int_X|S|_{h_p}^2\exp(-2tp\eta\varrho)\,\frac{\omega^n}{n!}\,\cdot\]
Let $\wi{P}_{t,p}$ be the Bergman kernel function of 
$H^0_{(2)}(X,L^p,\wi{h}_{t,p},\omega^n/n!)$.
Recall that $\|S\|_{p}$ is the norm given by the scalar product \eqref{e:ip}
on $H^0_0(X,L^p)$. Since $\varrho<0$ we have $\|S\|_{t,p}^2\geq \|S\|_{p}^2$
for any $S\in H^0_0(X,L^p)$.
Let $S\in H^0_0(X,L^p)$ with $\|S\|_{t,p}^2\leq1$. Then $\|S\|_{p}^2\leq1$, too, hence
\[|S|_{\wi{h}_{t,p}}^2=|S|_{h_p}^2\exp(-2tp\eta\varrho)
\leq P_{0,p}\exp(-2tp\eta\varrho)\,,\]
and thus
\begin{equation*}
\wi{P}_{t,p}\leq P_{0,p}\exp(-2tp\eta\varrho)\,.
\end{equation*}
Denote now by $P_p$ the Bergman kernel function of $H^0(X,L^p)$ endowed with the
scalar product \eqref{e:ip}. Since $H^0_0(X,L^p)$ is isometrically embedded
in $H^0(X,L^p)$ we have $P_{0,p}\leq P_p$. Consequently we have shown:
\begin{equation}\label{e:bv}
\begin{split}
\wi{P}_{t,p}\exp(2tp\eta\varrho)&\leq P_{0,p}\leq P_p\:\:\text{on $X$},\\
\wi{P}_{t,p}\leq P_{0,p}&\leq P_p\:\:\text{near $K$}.
\end{split}
\end{equation}
Let now $W$ be the neighborhood of $\Sigma$ defined in \eqref{e:Wgeom} and let $U_t$
be defined as in \eqref{e:ut}, so that the exponential estimate \eqref{e:estb1} 
holds on $U_t$ for $p>2t^{-1}$. By shrinking $U_t$ we can assume that $\eta=1$
on $U_t$. Setting $M:=e^{4\|h\|_\infty}A^t$ we obtain\eqref{e:exp1}.
By Theorem \ref{T:mt2} we have
\[\wi{P}_{t,p}(x)\geq(1-Cp^{-1/8})p^n\,\frac{c_1(L,\wi{h}_t)^n_x}{\omega^n_x}\] 
for every $p\in\N$ with $p\dist(x,\Sigma)^{8/3}>C$.
Note that $c_1(L,\wi{h}_t)\geq\delta\omega$ in the sense of currents on $X$. 
Since $c_1(L,\wi{h}_t)$ is smooth on $X\setminus\Sigma$ we have
$\dfrac{c_1(L,\wi{h}_t)^n}{\omega^n}\geq\delta^n$ on $X\setminus\Sigma$. 
By increasing $C$ if necessary,
it follows that
\[\wi{P}_{t,p}(x)\geq \frac{p^n}{C}\:\:\:\:\text{for $p>C\dist(x,\Sigma)^{-8/3}$}.\]
Hence 
\[P_{0,p}(x)\geq \frac{p^n}{C}\exp(2tp\varrho(x))\:\:\:\:
\text{for $x\in U_t$ and $p>C\dist(x,\Sigma)^{-8/3}$}.\]
This proves \eqref{e:exp2}.
In order to prove \eqref{e:exp3o} we need the following localization theorem
for the Bergman kernel for singular Hermitian metrics. 
We refer to \cite[Theorem 2.2]{AMM16} for a localization principle in the case of 
smooth metrics.
\begin{Theorem}\label{T:bkloc}
Let $(X,\omega)$ be a compact Hermitian manifold and $L\to X$ be a holomorphic
line bundle. Consider two singular Hermitian metrics $h_1$ and $h_2$ on $L$, which
are smooth outside a proper analytic set $\Sigma\subset X$ and such that 
$c_1(L,h_1)$, $c_1(L,h_2)$ are K\"ahler currents. Let $P_p^{(j)}$ 
be the Bergman projection on $H^0(X,L^p,h^p_j,\omega^n/n!)$, $j=1,2$.
We assume that there exists an open set $U\Subset X\setminus\Sigma$ 
such that $h_1=h_2$ on $U$. 
Then the Bergman kernels satisfy $P_p^{(1)}(z,w)-P_p^{(2)}(z,w)=O(p^{-\infty})$ on $U$ 
in any $\cC^\ell$-topology, $\ell\in\N$, as $p\to\infty$. 
\end{Theorem}
\begin{proof}
The proof follows essentially from the analysis in \cite{HsM:14} 
(see also \cite{HsM:16}).
Let $h_0$ be any singular Hermitian metric on $L$, smooth on $X\setminus\Sigma$
and satisfying $c_1(L,h_0)\geq\varepsilon\omega$ in the sense of currents on $X$,
for some $\varepsilon>0$.
Let $P_p^{(0)}$
be the Bergman projection on $H^0(X,L^p,h^p_0,\omega^n/n!)$.

Consider an open set $D\subset U$
such that $L|_D$ is trivial. Let $s:D\to L$ be a holomorphic frame
and let $\varphi\in\cC^\infty(D)$ be the weight of $h_0$ corresponding to $s$,
that is, $|s|_{h_0}=e^{-\varphi}$. 
Let us denote by $\cE'(D)$ the space of distributions with compact support
in $D$ and by $L^2(D)$ the space of square-integrable functions with respect to
the volume form $\omega^n/n!$.
The localized Bergman projection with respect to $s$ is the operator
$P_{p,s}^{(0)}:L^2(D)\cap\cE'(D)\to L^2(D)$, defined by 
$P_{p}^{(0)}(ue^{p\varphi}s^{\otimes p})=P_{p,s}^{(0)}(u)e^{p\varphi}s^{\otimes p}$.
It is easy to see that 
\begin{equation}\label{e:ps0}
P_{p}^{(0)}(z,w)=P_{p,s}^{(0)}(z,w)e^{p(\varphi(z)-\varphi(w))}
s^{\otimes p}(z)\otimes (s^{\otimes p})^*(w) \in L^p_z\otimes (L^p_w)^*\,,\:\:
z,w\in D.
\end{equation}
By \cite[Theorem\,9.2]{HsM:14} the kernel of $P_{p,s}^{(0)}$ satisfies
\begin{equation}\label{e:bl}
P_{p,s}^{(0)}(z,w)=\mathcal{S}_p(z,w)+O(p^{-\infty})\:\:\text{on $D$},
\end{equation}
where $\mathcal{S}_p$ is the localized approximate Szeg\H{o} kernel
defined in \cite[(3.43)]{HsM:14}. Note that by \cite[Theorem\,3.12]{HsM:14} we have
\begin{equation}\label{e:sp}
\mathcal{S}_p(z,w)=e^{ip\Psi(z,w)}b(z,w,p)+O(p^{-\infty})\:\:\text{on $D$},
\end{equation}
where $\Psi:D\times D\to\C$ is a phase function depending on the eigenvalues of $c_1(L,h_0)$
with respect to $\omega$ and described precisely in \cite[Theorem\,3.8]{HsM:14}.
Moreover, $b(\cdot,\cdot,p):D\times D\to\C$ is a semi-classical symbol of order $n=\dim X$,
depending only on the restriction of $h$ and $\omega$ to $D$.

We apply now these results for $h_0=h_1$ and $h_0=h_2$. Since $h_1|_D=h_2|_D$
we deduce that the weight $\varphi$, the phase $\Psi$ and the symbol 
$b(\cdot,\cdot,p)$ above are the
same for $h_1$ and $h_2$. We infer from \eqref{e:bl} and \eqref{e:sp} that
$P_{p,s}^{(1)}(z,w)-P_{p,s}^{(2)}(z,w)=O(p^{-\infty})$ on $D$.
Finally, \eqref{e:ps0} yields $P_{p}^{(1)}(z,w)-P_{p}^{(2)}(z,w)=O(p^{-\infty})$ on $D$.
The proof of Theorem \ref{T:bkloc} is complete.
\end{proof}
We apply now Theorem \ref{T:bkloc} to the metrics $\wi{h}_t$ and $h$, which are equal
on a neigborhood $V$ of $K$ and infer that
\begin{equation}\label{e:sp1}
\wi{P}_{t,p}- P_p=O(p^{-\infty})\:\:\text{locally uniformly on $V$}.
\end{equation}
Combined with \eqref{e:bv}, \eqref{e:sp1} yields \eqref{e:exp3o}.
Finally, \eqref{e:exp3} and \eqref{e:coeff} follow from the expansion
of the Bergman kernel $P_p$ (see \cite[Theorems\,4.1.1--3]{MM07})
or of the singular Bergman kernel (see \cite[Theorem\,1.8]{HsM:14}).

\medskip
\noindent
\textbf{\emph{Acknowledgments.}} We thank the referee for 
stimulating comments.


\def\cprime{$'$}
\providecommand{\bysame}{\leavevmode\hbox to3em{\hrulefill}\thinspace}
\providecommand{\MR}{\relax\ifhmode\unskip\space\fi MR }
\providecommand{\MRhref}[2]{%
  \href{http://www.ams.org/mathscinet-getitem?mr=#1}{#2}
}
\providecommand{\href}[2]{#2}

\end{document}